\documentclass[preprint,12pt,3p]{elsarticle}
\usepackage{index}
\usepackage{graphicx}
\usepackage{amsfonts,amsmath,amssymb,amsthm,enumerate,latexsym}

\newtheorem{theorem}{Theorem}
\newtheorem{lemma}[theorem]{Lemma}
\newtheorem{corollary}[theorem]{Corollary}
\newtheorem{observation}[theorem]{Observation}
\newtheorem{proposition}[theorem]{Proposition}

\theoremstyle{definition}
\newtheorem{definition}[theorem]{Definition}

\def\picture#1#2#3{\begin{figure}
\begin{center}
\includegraphics{#2}
\caption{#1}
\def\tmp{#3}
\ifx\tmp\empty\else\label{#3}\fi
\end{center}
\end{figure}}

\def\Rplus{R^{+}}
\def\Rminus{R^{-}}
\def\M{{\mathcal M}}
\def\CSP{\operatorname{CSP}}
\def\Datalog{\operatorname{DATALOG}}
\journal{European Journal of Combinatorics}

\begin{document}

\begin{frontmatter}
\title{Maltsev digraphs have a majority polymorphism}

\tnotetext[abbr]{Abbreviations used: CSP (Constraint Satisfaction Problem) }

\author[uk]{Alexandr Kazda}
\ead{alexak@atrey.karlin.mff.cuni.cz.}
\address[uk]{{Department of Algebra, Charles University,
Sokolovsk\'a 83, 186 75, Praha 8, Czech Republic}}

\begin{abstract}
We prove that when a digraph $G$ has a Maltsev
polymorphism, then $G$ also has a majority polymorphism. We consider the
consequences of this result for the structure of Maltsev digraphs
and the complexity of the Constraint Satisfaction Problem.

\end{abstract}

\begin{keyword}
digraph 
\sep
Maltsev polymorphism
\sep
 majority
\sep
CSP
\sep
 Datalog

\MSC[2010] 08A70\sep 05C99 \sep 68Q17
\end{keyword}

\end{frontmatter}
%
%
%
%
%

\section{Introduction}
In recent years, a marriage of universal algebra with graph theory brought
about great advances in the study of the Constraint Satisfaction Problem (CSP) and
related areas (see \cite{BJK} or \cite{libor}). 

Given $G$, the problem $\CSP(G)$ with input $H$ consists of deciding whether there exists
a homomorphism from $H$ to $G$. (Note that in the general theory, $G$ and $H$
can be any relational structures, however we will consider only digraphs in
this paper.) An important open question is how to determine the complexity of
$\CSP(G)$ from the properties of $G$. In particular, the famous \emph{dichotomy
conjecture} by Feder and Vardi claims that if $\CSP(G)$ is not polynomial time
solvable, then it is NP-complete (see \cite{FV}). 

At the core of universal algebra's success in describing the complexity of CSP
is the focus on algebras of polymorphisms. 
It turns out that the more polymorphisms $G$ admits, the easier it
is to solve $\CSP(G)$. More precisely, if we have two digraphs $G$ and $G'$ on
the same vertex set and the algebra of polymorphisms of $G'$
is contained in the algebra of polymorphisms of $G$, then $\CSP(G)$ can be
reduced to $\CSP(G')$ in logarithmic space (see Theorem 2.16 in \cite{BJK} for
the idea, \cite{jeavons} for the logspace reduction proper). 

Two often-encountered kinds of polymorphisms are the Maltsev and majority
polymorphism. Existence of either kind of polymorphism guarantees a polynomial
time algorithm for $\CSP(G)$ (for majority, see eg. \cite{DalmauKrokhin}, 
for the Maltsev polymorphism, see \cite{Dalmau}), while
if $G$ has both these polymorphisms then $\CSP(G)$ is even solvable in
deterministic logarithmic space (see \cite{DalmauLarosse}). 
The purpose of this paper is to prove that whenever a digraph $G$ has a Maltsev
polymorphism, then $G$ has a majority polymorphism as well. 
The core idea of our proof is that, given a Maltsev digraph $G$, we factorize
$G$, obtain majority on the factorgraph $G^+$ by induction and then
extend the majority to the original $G$. 

We give an overview of the implications of our result for CSP in the
Conclusions section.

\section{Preliminaries}
While our solution is rather elementary, the reader might still benefit from
understanding the context in which we wrote this paper. A good summary of the
combinatorics of digraph homomorphism can be found in \cite{nesetril}, while
\cite{BJK} presents an overview of the algebraic techniques in CSP and
\cite{burris} provides a good general introduction to universal algebra. 

Through the paper, \emph{digraph} will mean a finite directed graph with loops allowed.
We will allow the null digraph, however the main result does not change if we
demand that $V(G)\neq \emptyset$.

\begin{definition}
Let $G$ be a digraph. The mapping $f:V(G)^n\to V(G)$ is a
\emph{polymorphism} if it is true that whenever we have $(u_1,v_1)$,
$(u_2,v_2)$, \dots, $(u_n,v_n)\in
E(G)$ then 
\[
(f(u_1,\dots,u_n),f(v_1,\dots,v_n))\in E(G).
\]
In this situation, we
will also say that the mapping $f$ is \emph{compatible} with the edge relation $E(G)$.
\end{definition}

\begin{definition}
Let $G,H$ be digraphs. A mapping $f:V(G)\to V(H)$ is called a \emph{homomorphism}
if whenever $(u,v)\in E(G)$, we have $(f(u),f(v))\in E(H)$.
\end{definition}

Given a set of relations $R_1,\dots R_n$, a primitive positive definition of a
relation $S$ is any formula using only variables, existential
quantification, the equality relation, relations $R_1,\dots, R_n$ and
conjunctions. The following 
proposition is a part of the folklore of CSP (we invite the readers to prove
the proposition as an exercise).
\begin{proposition}
Let $G$ be a digraph and $R$ a relation defined by a primitive positive formula using the
relation $E(G)$. Then all polymorphisms of $G$ are compatible with $R$.
\end{proposition}

\begin{definition}
A \emph{Maltsev polymorphism} of a digraph $G$ is any ternary polymorphism $m$ such that the following
equalities hold for all $x,y\in V(G)$:
\begin{align*}
m(x,y,y)&=x\\
m(x,x,y)&=y
\end{align*}
\end{definition}

\begin{definition}
A \emph{majority} is any polymorphism $M$ such that the following
equalities hold for all $x,y\in V(G)$:
\begin{align*}
M(x,y,y)&=y\\
M(y,x,y)&=y\\
M(y,y,x)&=y.
\end{align*}
\end{definition}

\begin{definition}
A digraph $G$ is \emph{Maltsev} if it has a Maltsev polymorphism. We say that $G$
\emph{has a majority} if there exists a majority polymorphism of $G$.
\end{definition}

During the proof we will need the following notation: Let $v$ be a vertex in a
digraph $G$. We then denote by $v^+$ the vertex set $\{u\in V(G):(v,u)\in E(G)\}$
and, similarly, by $v^-$ the vertex set $\{u\in V(G):(u,v)\in E(G)\}$. We will
occasionally extend the mappings $v^+$ and $v^-$ to whole sets of vertices.

We will call a vertex $v$ a \emph{source} if $v^-=\emptyset$ and a \emph{sink}
if $v^+=\emptyset$. If $v$ is neither a source nor a sink, we will call $v$
\emph{smooth}. For $G$ a digraph, let $S^-(G)$ be the set of all sources of $G$ and
$S^+(G)$ be the set of all sinks in $G$.

\section{Maltsev digraphs have majority}

We begin with an easy but fundamental observation:
\begin{observation}\label{obsCrucial}
Let $G$ be a Maltsev digraph. If $x,x',y,y'$ are (not necessarily all different)
vertices of $G$ and $(x,y),(x',y),(x',y')\in E(G)$ then also $(x,y')\in E(G)$ (see
Figure~\ref{figCrucial}).
\end{observation}
\begin{proof}
Let $m$ be the Maltsev polymorphism of $G$. From the definition of a polymorphism we
get $(m(x,x',x'),m(y,y,y'))\in E(G)$, but $m(x,x',x')=x$ and $m(y,y,y')=y'$, so
$(x,y')\in E(G)$.
\end{proof}
\picture{The vertices and edges in Observation~\ref{obsCrucial}.}{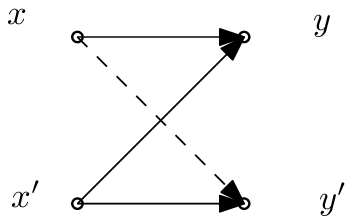}{figCrucial}

Motivated by this observation we give the following definition:
\begin{definition}
Call a digraph $G$ \emph{rectangular} if 
whenever $(x,y),(x',y),(x',y')\in E(G)$ then also $(x,y')\in E(G)$.
\end{definition}

All Maltsev digraphs are rectangular, however there are rectangular digraphs
that are not Maltsev (see Figure~\ref{figCounterexample}; we will later see that this
digraph violates Observation~\ref{obsPlus}).

\picture{A rectangular digraph that is not Maltsev.}{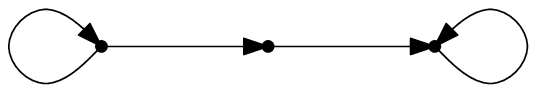}{figCounterexample}

\begin{definition}
Let $G$ be a digraph. We define two relations $\Rminus,\Rplus$ on $V(G)$ as follows:
\begin{align*}
x\Rminus y&\Leftrightarrow \exists z, (z,x),(z,y)\in E(G)\\
x\Rplus y&\Leftrightarrow \exists z, (x,z),(y,z)\in E(G).
\end{align*}
\end{definition}

Observe that the relations $\Rplus$ and $\Rminus$ are symmetric. Also, the
definition of $\Rplus$ and $\Rminus$ is a primitive positive one, so any
polymorphism of $G$ is compatible with $\Rplus$ and $\Rminus$.

The following lemma is actually a collection of easy observations:
\begin{lemma}\label{lemList}
Let $G$ be a rectangular digraph. Then the following holds:
\begin{enumerate}
\item If $v$ is a sink then there is no $x$ such that $x\Rplus v$.
\item If $v$ is a source then there is no $x$ such that $x\Rminus v$.
\item $\Rplus$ is an equivalence relation on $G\setminus S^+(G)$.
\item $\Rminus$ is an equivalence relation on $G\setminus S^-(G)$.
\item Whenever $x\Rplus y$, we have $x^+=y^+$ and $x^+$ is an equivalence class
of $\Rminus$.
\item Whenever $x\Rminus y$, we have $x^-=y^-$ and $x^-$ is an equivalence
class of $\Rplus$.
\item The mapping $\phi:X\mapsto X^+$ is a bijection from the set of equivalence
classes of $\Rplus$ to the set of equivalence classes of $\Rminus$.
\end{enumerate}
\end{lemma}
\begin{proof}
Parts (1) and (2) are easy: If $v^+=\emptyset$, there is no $z$ such that
$(v,z)\in E(G)$ and therefore $v$ can not be $\Rplus$-related to anything.
Similarly for the dual case.

We know that the relation $\Rplus$ is
symmetric on $V(G)\setminus S^+$. To prove reflexivity, consider $x\in
V(G)\setminus S^+$. As $x\not\in S^+$, there
exists $z\in V(G)$ such that $(x,z)\in E(G)$ and so $x\Rplus x$. 

From Observation~\ref{obsCrucial}, it follows that whenever $x\Rplus y$, we
have $x^+=y^+\neq\emptyset$. From this we can easily get transitivity of
$\Rplus$: If $x\Rplus y\Rplus z$, we have $x^+=z^+\neq\emptyset$ and so there
exists $t$ such that $(x,t),(z,t)\in E(G)$. Again, the proof of (4) is similar.

We already have half of (5), it remains to show that whenever
$x^+\neq\emptyset$, the set $x^+$ is an equivalence class of $\Rminus$.
Obviously, $x$ is a witness that all the vertices of $x^+$ are
$\Rminus$-related. If now $u\in x^+$ and $v\Rminus u$, then
Observation~\ref{obsCrucial} gives us that $(x,v)\in E(G)$ and so $v\in x^+$,
concluding the proof. Once more, the statement (6) is a dual version of (5).

To prove (7), observe that if $X$ is an equivalence class of $\Rplus$ then
$X=\left(X^+\right)^-$ and similarly whenever $Y$ is an equivalence class of
$\Rminus$, we have $Y=\left(Y^-\right)^+$. The mapping $\phi$ is invertible and
therefore is a bijection.
\end{proof}

Let $G$ be a rectangular digraph. Denote by $G^+$ the digraph whose vertices are the
equivalence classes of $\Rplus$ with $(X,Y)\in E(G^+)$ iff there exist vertices $x\in
X$ and $y\in Y$ in $V(G)$ such that $(x,y)\in E(G)$. Similarly, let $G^-$ be the 
digraph whose vertices are the $\Rminus$ equivalence classes with 
$(X,Y)\in E(G^-)$ iff there exist vertices $x\in X$ and $y\in Y$ such that
$(x,y)\in E(G)$.

\begin{lemma}\label{lemIsomorphism}
Let $G$ be a rectangular digraph and $\phi$ the mapping from part (7) of
Lemma~\ref{lemList}. Then  the mapping 
$\phi$ is an isomorphism of $G^+$ to $G^-$.
\end{lemma}
\begin{proof}
See Figure~\ref{figIsomorphism} for a picture of the proof.

First observe that we have $(X,Y)\in E(G^+)$ iff $X^+\cap Y\neq\emptyset$ in $G$. But
$X^+=\phi(X)$, therefore $XY\in E(G^+)$ iff $\phi(X)\cap Y\neq\emptyset$.
Similarly, $XY\in E(G^-)$ iff $X\cap \phi^{-1}(Y)\neq\emptyset$.

We know that $\phi$ is a bijection from $V(G^+)$ onto $V(G^-)$. We need to show
that $(X,Y)\in E(G^+)$ iff $(\phi(X),\phi(Y))\in E(G^-)$. However, we already have a chain
of equivalent statements: 
\begin{align*}
(X,Y)\in E(G^+)\\
\phi(X)\cap Y\neq\emptyset\\
\phi(X)\cap \phi^{-1}(\phi(Y))\neq\emptyset\\
(\phi(X),\phi(Y))\in E(G^-),
\end{align*}
which is precisely what we wanted.
\end{proof}

So far we have used only rectangularity. However, the following
observation is not true for rectangular digraphs (try it for the digraph in
Figure~\ref{figCounterexample}).

\begin{observation}\label{obsPlus}
Let $G$ be a Maltsev digraph. Then $G^+$ is also Maltsev.
\end{observation}
\begin{proof}
Consider the Maltsev polymorphism $m$ of $G$. Define the map $t$ on $G^+$ by
letting
\[
t(x/_{\Rplus},y/_{\Rplus},z/_{\Rplus})=m(x,y,z)/_{\Rplus}
\]
for $x,y,z$ vertices in $V(G)\setminus S^+$.

As the operation $m$ is compatible with the relation $\Rplus$, $t$ is
well-defined. Moreover, $t$ satisfies the Maltsev equations and a little thought
gives us that $t$ is a polymorphism of $G^+$. Therefore, $G^+$ is Maltsev.
\end{proof}

\picture{Picture proof of
Lemma~\ref{lemIsomorphism}.}{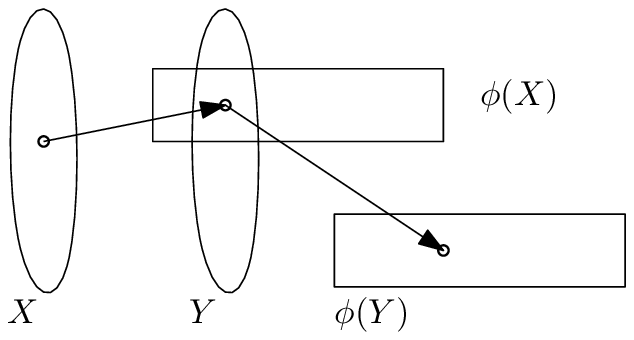}{figIsomorphism}
We are now ready to prove Theorem~\ref{thm1}.

\begin{theorem}\label{thm1}
Any Maltsev digraph has a majority polymorphism.
\end{theorem}

\begin{proof}
Let $H$ be a vertex-minimal digraph such that $H$ is Maltsev but has
no majority. We will show that this leads to a contradiction.

Let us first consider the case $|V(H^+)|=|V(H^-)|=|V(H)|$. This is only possible when $H$ has no
sources or sinks and every $\Rplus$ or $\Rminus$-class of $H$ is a singleton.
From Lemma~\ref{lemList} we obtain that $H$ is then the digraph of the permutation
$\phi$ and therefore is a disjoint union of directed cycles (we consider the
null digraph to be an empty union of directed cycles). It is easy to
verify that the mapping $M$ defined as
\[
M(x,y,z)=
\begin{cases}
y&\text{if $y=z$,}\\
x&\text{else,}\\
\end{cases}
\]
is a majority polymorphism of $H$.

We can thus assume that $|V(H^+)|<|V(H)|$. As $H$ is the smallest
counterexample and $H^+$ is Maltsev by Observation~\ref{obsPlus}, there exists a majority polymorphism $M^+$
of $H^+$. Denote by $M^-$ the polymorphism of $H^-$ conjugated to $M^+$ via
$\phi$, i.e. 
\[
M^-(x,y,z)=\phi(M^+(\phi^{-1}(x),\phi^{-1}(y),\phi^{-1}(z))).
\]

We now want to find a map $M(x,y,z)$ on $H$ so that the following holds:
\begin{enumerate}
\item $M(x,x,y)=M(x,y,x)=M(y,x,x)=x$.
\item If $x,y,z \not\in S^+$ then
$M(x,y,z)/_{\Rplus}=M^+(x/_{\Rplus},y/_{\Rplus},z/_{\Rplus})$.
\item If $x,y,z \not\in S^-$ then
$M(x,y,z)/_{\Rminus}=M^-(x/_{\Rminus},y/_{\Rminus},z/_{\Rminus})$.
\end{enumerate}
We will later prove that any such $M$ is a polymorphism, concluding the proof.
However, we will not explicitly demand $M$ to be a polymorphism for now.

We can construct $M$ by choosing, for each triple $(x,y,z)\in V(G)^3$ an image that
satisfies (1)--(3). However, we need to show that the candidate set is nonempty
for each choice of $(x,y,z)$.
 
As $M^+$ and $M^-$ are majority polymorphisms, the equalities (2) and (3)
follow from (1) whenever two of the variables $x,y,z$ are the same. Therefore, 
the only way that the value $M(x,y,z)$ can fail to exist is if for some 
$x_1,x_2,x_3$ vertices in $V(G)\setminus\left( S^+(G)\cup S^-(G)\right)$ we would have
\[
M^+(x_1/_{\Rplus},x_2/_{\Rplus},x_3/_{\Rplus})
\cap
M^-(x_1/_{\Rminus},x_2/_{\Rminus},x_3/_{\Rminus})=\emptyset.
\]

Can such a thing happen? We know that for $i=1,2,3$ the set 
$x_i/_{\Rplus}\cap x_i/_{\Rminus}$ is nonempty. Therefore, we have
\[
\phi(\phi^{-1}(x_i/_{\Rminus})) \cap x_i/_{\Rplus} \neq\emptyset
\]
and, as in the proof of Lemma~\ref{lemIsomorphism}, we obtain that
\[
\left(\phi^{-1}(x_i/_{\Rminus}),x_i/_{\Rplus}\right)\in E(H^+)
\]
for each $i=1,2,3$. Applying the polymorphism $M^+$, we obtain
\[
\left(M^+(\phi^{-1}(x_1/_{\Rminus}),\phi^{-1}(x_2/_{\Rminus}),\phi^{-1}(x_3/_{\Rminus})),M^+(x_1/_{\Rplus},x_2/_{\Rplus},x_3/_{\Rplus})\right)\in
E(H^+),
\]
but this is precisely the same as
\[
\phi(M^+(\phi^{-1}(x_1/_{\Rminus}),\phi^{-1}(x_2/_{\Rminus}),\phi^{-1}(x_3/_{\Rminus})))\cap
M^+(x_1/_{\Rplus},x_2/_{\Rplus},x_3/_{\Rplus})\neq\emptyset.
\]
Now recall the definition of $M^-$ to see that we have just shown that
\[
M^-(x_1/_{\Rminus},x_2/_{\Rminus},x_3/_{\Rminus})
\cap
M^+(x_1/_{\Rplus},x_2/_{\Rplus},x_3/_{\Rplus})\neq\emptyset.
\]
Therefore, there exists a map $M$ satisfying conditions (1)--(3). By (1) this
$M$ satisfies the majority equations. It
remains to show that this $M$ is in fact a polymorphism of $H$.

Let $(x,x'),(y,y'),(z,z')\in E(H)$. We want to show that
$(M(x,y,z),M(x',y',z'))\in E(H)$.
Obviously $x,y,z$ are not sinks and $x',y',z'$ are not sources. Moreover, a
little thought gives us that (see Figure~\ref{figObvious}):
\[
\phi^{-1}(x'/_{\Rminus})=x/_{\Rplus},\quad
\phi^{-1}(y'/_{\Rminus})=y/_{\Rplus},\quad
\phi^{-1}(z'/_{\Rminus})=z/_{\Rplus}.
\]
\picture{Showing that $M$ is a polymorphism.}{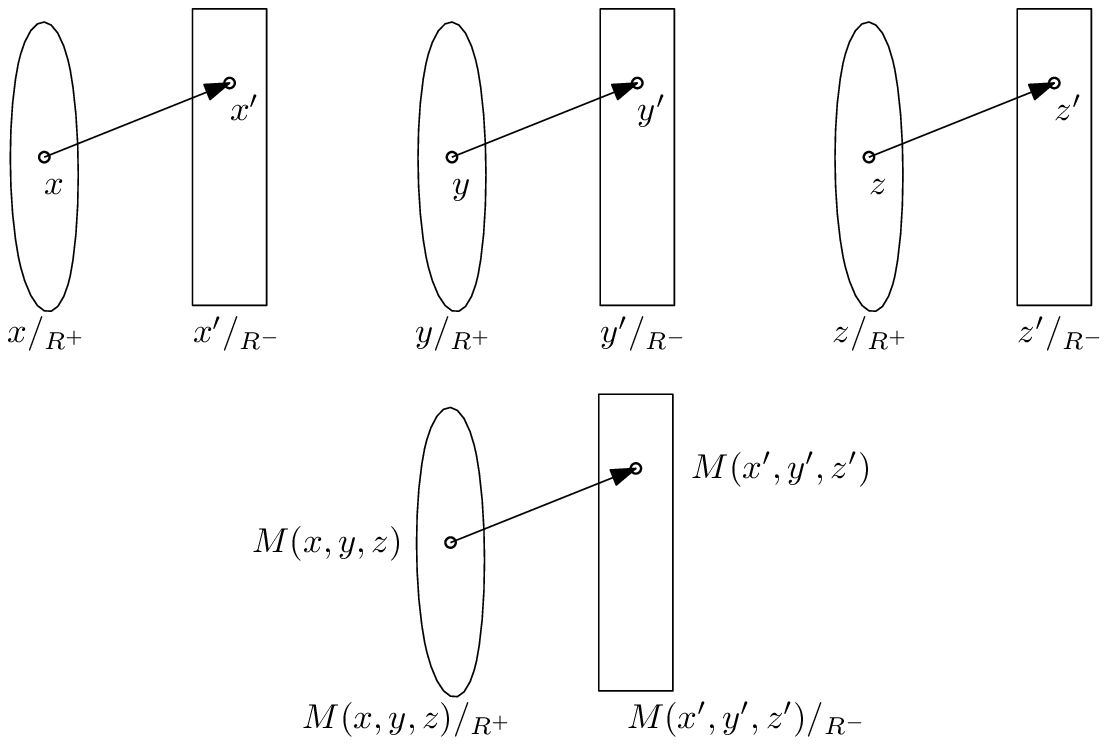}{figObvious}
But then 
\begin{align*}
M(x',y',z')/_{\Rminus}&=M^-(x'/_{\Rminus},y'/_{\Rminus},z'/_{\Rminus})\\
&=\phi(M^+(\phi^{-1}(x'/_{\Rminus}),\phi^{-1}(y'/_{\Rminus}),\phi^{-1}(z'/_{\Rminus})))\\
&=\phi(M^+(x/_{\Rplus},y/_{\Rplus},z/_{\Rplus})).
\end{align*}
Now observe that actually
$M(x,y,z)^+=\phi(M^+(x/_{\Rplus},y/_{\Rplus},z/_{\Rplus}))$. Putting the last
two equalities together, we obtain $M(x,y,z)^+=M(x',y',z')/_{\Rminus}$ which can only happen when we
have $(M(x,y,z),M(x',y',z'))\in E(H)$, concluding the proof.
\end{proof}

It is straightforward to translate Theorem~\ref{thm1} to the language of
universal algebra varieties (see \cite{burris} for background and details):
\begin{corollary}\label{corHSP}
If $V$ is a variety generated by the algebra of all polymorphisms of some
digraph $G$ then $V$ is congruence permutable iff $V$ is arithmetic.
\end{corollary}

Let us close this section with a description of the class of all Maltsev digraphs.

\begin{lemma}\label{lemIff}
Let $G$ be a rectangular digraph. Then $G$ is Maltsev iff $G^+$ is Maltsev.
\end{lemma}
\begin{proof}
We already know the ``$\Rightarrow$'' implication from Observation~\ref{obsPlus}.

On the other hand, if $m^+$ is a Maltsev polymorphism of $G^+$, we can use a
construction similar to the one from the proof of Theorem~\ref{thm1} to obtain
a Maltsev polymorphism $m$ of $G$.
\end{proof}

From this lemma we see that if we start from a
disjoint union of directed cycles $G_0$ and then in each step choose a
rectangular digraph $G_{i+1}$ so that $\left(G_{i+1}\right)^+=G_i$, all the
graphs $G_0,G_1,\dots$ will be Maltsev. Moreover, every Maltsev digraph can
be obtained in this way (with a suitable choice of the sequence $G_0,G_1,\dots,
G_n$) because every Maltsev digraph becomes a
disjoint union of directed cycles ofter applying the $^+$ operation
sufficiently many times. 

Let us state our findings in a more compact form:

\begin{corollary}\label{corCharacterise}
The class of all Maltsev digraphs $\M$ is the smallest class of digraphs such that:
\begin{enumerate}
\item All digraphs in $\M$ are rectangular,
\item $\M$ contains all disjoint unions of directed cycles and all edgeless
digraphs,
\item $\M$ is closed under taking the preimages under the map $G\mapsto G^+$
(i.e. if $H\in \M$ and $G$ is rectangular such that $G^+=H$ then $G\in\M$).
\end{enumerate}
\end{corollary}

Note: We explicitly mention edgeless digraphs in part
(2) so that the corollary is true even if we disallow the null digraph.

\section{Maltsev digraphs and the CSP}
We conclude our paper with a note about the connections with the Constraint
Satisfaction Problem. However, we
must first introduce two new notions: adding constants and the Datalog language.

Observe that both majority and Maltsev polymorphisms preserve the unary
constant relation $c_v=\{(v)\}$ for every $v\in V(G)$ (because
$m(v,v,v)=M(v,v,v)=v$).
Therefore, we can ``enhance'' any Maltsev digraph $G$ by adding one constant relation for every
$v\in V(G)$. Call the resulting relational structure $G_c$. Observe that
$\CSP(G_c)$ is essentially the problem of determining whether a given partial
mapping $V(H)\to V(G)$ can be extended to a digraph homomorphism $H\to G$. 
It is not difficult to observe that $\CSP(G_c)$ is at least as hard as
$\CSP(G)$. To make our results more meaningful, we will now be talking about the
complexity of $\CSP(G_c)$.

There is an important class of CSPs that can be solved using the Datalog
language or some subset thereof (note that ``can be solved'' actually means
that the \emph{complements} of these problems lie in the $\Datalog$ class,
however that is just a technicality).  Another name for such CSPs is
problems of \emph{bounded width} (see \cite{libor} for an overview).

Putting Theorem~\ref{thm1} together with known body of knowledge about Datalog, we obtain that if $G$
is a Maltsev digraph then $\CSP(G_c)$ can be solved using a rather simple kind
of consistency test in logarithmic space.
 
As shown in \cite{DalmauKrokhin}, if $G$ admits a majority polymorphism, then $\CSP(G_c)$
can be solved using linear Datalog (in nondeterministic logarithmic space). On the
other hand, Maltsev polymorphism in general relational structure $G$ does
\emph{not} guarantee that there is a Datalog solution to $\CSP(G_c)$. However, if
$G$ is actually a digraph, then Maltsev implies majority by Theorem~\ref{thm1}
and hence there exists a linear Datalog solution. We can improve this statement further,
as \cite{DalmauLarosse} tells us that in this case it is enough to use the
so-called symmetric Datalog, ensuring that $\CSP(G_c)$ is solvable in \emph{deterministic}
logarithmic space. 

\section{Conclusions and open problems} 
Digraphs are a rather versatile structures that can often ``emulate'' other
structures in various ways (see eg. \cite{nesetril} or Section 5 of \cite{FV}).
Our result, however, shows that sometimes digraphs are not general enough: In
general relational structures, Maltsev and majority polymorphisms are
independent of each other while in digraphs Maltsev implies majority.

Therefore, we would like to know what makes digraphs behave like this. And,
perhaps more importantly, what other implications of this kind (i.e. ``If $G$
has a polymorphism $s$ then $G$ has a polymorphism $t$.'') hold for digraphs but
not for general relational structures? 

Finally, two more direct future tasks spring to mind: First, to characterize all
Maltsev digraphs in a more explicit way than Corollary~\ref{corCharacterise} and, returning to combinatorics, to
count the number of all Maltsev digraphs on $n$ vertices (or at least give some
asymptotics).

\section{Acknowledgements}
Supported by the GA\v CR project GA\v CR 201/09/H012. Thanks to 
Jakub Bul\'\i n, Micha\l{} Stronkowski and especially to Libor Barto for valuable input.

\bibliographystyle{amsplain}
\bibliography{citations}
\end{document}